\documentclass{article}
\usepackage[utf8]{inputenc}
\usepackage{amsmath,amsthm,dsfont,pifont,amsfonts,MnSymbol}

\title{Sequence-covering maps on submetrizable spaces 
\footnote{The work was performed as part of research conducted in the Ural Mathematical Center with the financial support of the Ministry of Science and Higher Education of the Russian Federation.}
\footnote{Keywords\textup{:} submetrizable space, sequence-covering map, 1-sequence-covering map, closed map, compact map}}
\author{Vlad Smolin \\ \small{Krasovskii Institute of Mathematics and Mechanics, Ural Federal University,} \\ \small{ Yekaterinburg, Russia} \\ \small{e-mail: SVRusl@yandex.ru}}

\theoremstyle{plain}
\newtheorem{teo}{Theorem}
\newtheorem{lemm}[teo]{Lemma}

\newtheorem{ques}[teo]{Question}

\theoremstyle{definition}

\newtheorem{nota}[teo]{Notation}

\begin{document}

\maketitle
\begin{abstract}
A topological space is called a submetrizable if it can be mapped onto a metrizable topological space by a continuous one-to-one map. In this paper we answer two questions concerning sequence-covering maps on submetrizable spaces.
\end{abstract}

\section{Introduction}
In \cite{Lin4} authors proved that if $X$ is metrizable and $f$ is a compact and sequence-covering map on $X$, then $f$ is 1-sequence-covering. Also in \cite{Lin5} it was proved that each closed sequence-covering map on a metrizable space is an 1-sequence-covering map. This results motivate questions which were posed in \cite{Lin1}:
\begin{ques}
Are the sequence-covering and compact maps on a submetrizable space 1-sequence-covering?
\end{ques}
\begin{ques}
Are the closed sequence-covering maps on a submetrizable space 1-sequence-covering?
\end{ques} 
We give a negative answer for the first question and a positive for the second. 
\section{Notation and terminology}
We use terminology from \cite{enc}. Throughout this paper all topological spaces are assumed to be Hausdorff, all maps are continuous and onto. 

\begin{nota}\label{not01}%
    Let $\langle X, \tau_X \rangle$ and $\langle Y, \tau_Y \rangle$ be a topological spaces and let $f\colon X \to Y$ be a map. Then
    \begin{itemize}
        \item $\mathsurround=0pt
 \omega$ $\coloneq$ the set of finite ordinals $=$ the set of natural numbers, so $0=\varnothing\in\omega$ and ${n}=\{0,\ldots,{n}-1\}$ for all ${n}\in\omega;$
        \item $\mathsurround=0pt
 {f}\upharpoonright A$ $\coloneq$ the restriction of the function ${f}$ to the set ${A};$
        \item $f[A] \coloneq \{f(a)\ \colon\ a \in A\}$;
        \item $\tau_X {\upharpoonright} A := \{U \cap A: U \in \tau_X\} = $ the subspace topology of $A$;
        \item ${g}\circ {f}$ is the composition of functions ${g}$ and ${f}$ (that is, ${g}$ after ${f}$); 
        \item $\mathsurround=0pt
 {}^{B}\!{A}$ $\coloneq$ the set of functions from ${B}$ to ${A}$;
         \item $S_\omega \coloneq$ the sequential fan (see \cite[Example 1.8.7]{Lin2});
        \item $\mathsf{ConvSeq}(\tau_{X}, x) \coloneq $ the set of all non-trivial sequences in $X$ that converge to $x$;
        \item $\mathsf{ConvSeq}(\tau_{X}) \coloneq \bigcup_{x \in X} \mathsf{ConvSeq}(\tau_{X}, x)$;
        \item let $q \in {}^\omega X$ and $p \in {}^\omega Y$. We say that $q$ covers $p$ if $f(q(n)) = p(n)$ for all $n \in \omega$;
        \item $f$ is a sequence-covering map if for all $p \in \mathsf{ConvSeq}(\tau_{Y})$ there is $q \in \mathsf{ConvSeq}(\tau_{X})$ such that $q$ covers $p$; 
        \item $f$ is a 1-sequence-covering map if for all $y \in Y$ there is $x \in f^{-1}(y)$ such that for all $p \in \mathsf{ConvSeq}(\tau_{Y}, y)$ there is $q \in \mathsf{ConvSeq}(\tau_{X}, x)$ such that $q$ covers $p$;
        \item suppose that $A_1$ and $A_2$ are disjoint infinite subsets of $\omega$ and $g\colon A_1 \to A_2$ is the unique order isomorphism, then $A_1 < A_2\ \colon\longleftrightarrow\ \forall n \in A_1\ [n < g(n)]$.
    \end{itemize}
\end{nota}
\begin{lemm} \label{decomp}
    There exists a decomposition $\{A_i\ \colon\ i\in\omega\}$ of $\omega$ such that $A_i < A_j$ whenever $i < j$. \hfill\qed
\end{lemm}

\begin{lemm} \label{tech_lemm}
    Suppose $\{A_i\ \colon\ i\in\omega\}$ is a decomposition from Lemma \ref{decomp}, $f_{i,j}$ is the unique order isomorphism from $A_i$ onto $A_j$ for all $i,j \in \omega$, and $f_i$ is the unique order isomorphism from $A_i$ onto $\omega$ for all $i \in \omega$. If $a < b \in \omega$ and $c < d \in \omega$, then $f^{-1}_a(c) < f^{-1}_b(d)$.
\end{lemm}

\begin{proof}
    Since $a < b$, we have $A_a < A_b$. Note that $f_{a,b}\circ f^{-1}_a = f^{-1}_b$, and so $f^{-1}_a(c) < f_{a,b}(f^{-1}_a(c)) = f^{-1}_b(c) < f^{-1}_b(d)$.
\end{proof}

\section{A theorem and an example}

\begin{teo}
    Let $f\colon \langle X, \tau_{X} \rangle\to \langle Y, \tau_{Y} \rangle$ be a closed sequence-covering map, where $\langle X, \tau_{X} \rangle$ is a submetrizable space. Then $f$ is 1-sequence-covering.
\end{teo}

\begin{proof}
    Suppose that $\sigma$ is a metrizable topology on $X$ such that $\sigma \subseteq \tau_{X}$. Fix $y \in Y$. For every sequence $p \in \mathsf{ConvSeq}(\tau_{Y}, y)$ let 
    $$
    {\bf F}(p) \coloneq \{x \in X\ \colon\ \exists q \in \mathsf{ConvSeq}(\tau_{X}, x) \text{ such that } q \text{ covers } p\}.
    $$

    Since $f$ is continuous and sequence-covering, we see that if $p \in \mathsf{ConvSeq}(\tau_{Y}, y)$, then
    $$
        \varnothing \neq {\bf F}(p) \subseteq f^{-1}(y).
    $$

    Take sequences $p_1, p_2 \in \mathsf{ConvSeq}(\tau_{Y}, y)$. Define $p_3$ as follows:
    \[
        p_3(n) \coloneq 
        \begin{cases}
        p_1(\frac{n}{2}), & \text{if } n \text{ is even;} \\
        p_2(\frac{n-1}{2}), & \text{if } n \text{ is odd.}
        \end{cases}
    \]
    Then $p_3 \in \mathsf{ConvSeq}(\tau_{Y}, y)$ and ${\bf F}(p_3) \subseteq {\bf F}(p_1) \cap {\bf F}(p_2)$, and so 
    \begin{equation} \label{centered}
        \{{\bf F}(p)\ \colon\ p \in \mathsf{ConvSeq}(\tau_{Y}, y)\} \text{ is a centered family.}
    \end{equation}

    Now we prove that 
    $$
        \forall p \in \mathsf{ConvSeq}(\tau_{Y}, y)\ [{\bf F}(p) \text{ is closed in } \langle X, \sigma \rangle]. 
    $$
    Take $p \in \mathsf{ConvSeq}(\tau_{Y}, y)$. Suppose that $x \in X$ is an accumulation point of ${\bf F}(p)$ in $\langle X, \sigma \rangle$. Let $\{O_{n}(x)\ \colon\ n\in\omega\}$ be a local base at $x$ in $\langle X, \sigma \rangle$. Since $x$ is an accumulation point of ${\bf F}(p)$ in $\langle X, \sigma \rangle$, it is easy to see that
    $$
        \forall n \in \omega\ \exists k \in \omega\ \forall i \geq k\ [O_n(x) \cap f^{-1}(p_i) \neq \varnothing].
    $$
    And so there exists a sequence $q \in \mathsf{ConvSeq}(\sigma, x)$ such that $q$ covers $p$. Let us prove that $q \in \mathsf{ConvSeq}(\tau_{X}, x)$. Assume that there exists an infinite closed in $\langle X, \tau_{X} \rangle$ subset $B$ of $\{q_n\ \colon\ n \in \omega\}$, then $f[B]$ is a closed in $\langle Y, \tau_{Y} \rangle$ and infinite subset of $\{p_n \colon n \in \omega\}$, a contradiction. Then $q \in \mathsf{ConvSeq}(\tau_{X}, x)$, and so $x \in {\bf F}(p)$.

    Now we prove that
    \begin{equation} \label{count_comp}
        \forall p\in \mathsf{ConvSeq}(\tau_{Y}, y)\ [\langle {\bf F}(p), \sigma \upharpoonright {\bf F}(p) \rangle \text{ is a countably compact space}].
    \end{equation}
    Take a sequence $p\in \mathsf{ConvSeq}(\tau_{Y}, y)$. Assume that $\langle {\bf F}(p), \sigma \upharpoonright {\bf F}(p) \rangle$ contains a countable closed discrete subset $\{d_n\ \colon\ n\in\omega\}$. Since ${\bf F}(p)$ is closed in $\langle X, \sigma \rangle$, it follows that $\{d_n\ \colon\ n\in\omega\}$ is a closed discrete subset of $\langle X, \sigma \rangle$. Consequently, there exists a discrete family $\{U_n\ \colon\ n\in\omega\}$ of open in $\langle X, \sigma \rangle$ sets such that $d_n \in U_n$. For an arbitrary $n \in \omega$, since $d_n \in {\bf F}(p)$, there exists a sequence $q^n\in \mathsf{ConvSeq}(\tau_{X}, d_n)$ such that
    $q^n$ covers $p$. It follows that there exists an increasing sequence $\langle k_n \rangle_{n \in \omega} \in {}^\omega\omega$ such that $\{q^n_i\colon i \geq k_n\} \subseteq U_n$ for all $n \in \omega$. Consequently $\{q^n_{k_n}\ \colon\ n \in \omega\}$ is a closed discrete set, and so $f[\{q^n_{k_n}\colon n \in \omega\}] = \{p_{k_n} \colon n \in \omega\}$ is a closed discrete set, a contradiction.

    Since $\langle {\bf F}(p), \sigma \upharpoonright {\bf F}(p) \rangle$ is metrizable, from (\ref{count_comp}) it follows that
    \begin{equation} \label{comp}
        \forall p\in \mathsf{ConvSeq}(\tau_{Y}, y)\ [ \langle {\bf F}(p), \sigma \upharpoonright {\bf F}(p) \rangle \text{ is a compact space}].
    \end{equation}
    And so from (\ref{centered}) and (\ref{comp}) it follows that
    \begin{equation} \label{end}
        \bigcap \{{\bf F}(p)\ \colon\ p \in \mathsf{ConvSeq}(\tau_{Y}, y)\} \neq \varnothing.
    \end{equation}
    Since $y$ was arbitrary, from (\ref{end}) it follows that $f$ is 1-sequence-covering.  
\end{proof}

\begin{teo}
    There exist a countable submetrizable space $\langle X, \sigma \rangle$ and a compact sequence-covering map $f\colon \langle X, \sigma \rangle \to S_\omega$ such that $f$ is not 1-sequence-covering.
\end{teo}

\begin{proof}
    For every $n, m \in \omega$ and $g\in{}^\omega\omega$, let
    $$
        {\bf U}(n, m) \coloneq \{\langle n, k\rangle\ \colon\  k \geq m\};
    $$
    $$
        {\bf L}(g) \coloneq \bigcup_{n \in \omega} {\bf U}(n, g(n)).
    $$
    We define $S_\omega$ in a convenient form for us (for details see \cite[Example 1.8.7]{Lin2}). Let $X \coloneq \{0\} \cup \omega^2$ and let $\tau_{S_\omega}$ be the topology on $X$ generated by the base $\{\{x\}\ \colon\ x \in \omega^2\} \cup \{\{0\} \cup {\bf L}(g)\ \colon\ g\in{}^\omega\omega\}$, then $S_\omega \coloneq \langle X, \tau_{S_\omega} \rangle$.
    
    Now we turn to the definition of $\langle X, \sigma \rangle$. Let $\tau$ be the topology on $(\omega + 1)^2$ generated by the subbase 
    $$
        \tau_{(\omega + 1)^2} \cup \{\big((\omega+1)\times\{\omega\}\big) \cup{\bf L}(g)\ \colon\ g\in{}^\omega\omega\}.
    $$

    Let $\{A_i\ \colon\ i\in\omega\}$ be a decomposition of $\omega$ from Lemma \ref{decomp}, for every $k \in \omega$ denote by $\xi(k)$ the natural number $i$ such that $k \in A_i$. Also for every $i \neq j\in \omega$ by $f_i$ denote the unique order isomorphism from $A_i$ onto $\omega$ and by $f_{i, j}$ denote the unique order isomorphism from $A_i$ onto $A_j$.

    Define map $f^*\colon (\omega+1)^2 \to S_\omega$ by the following rules:
    $$
        f^*[(\omega+1)\times\omega] \coloneq 0;
    $$
    $$
        f^*[(\omega+1)\times\{k\}] \coloneq \langle \xi(k), f_{\xi(k)}(k) \rangle\ \forall k \in \omega.
    $$
    Now consider the subset of $(\omega + 1)^2$
    %$$
    %    X \coloneq \big((\omega+1)\times\{\omega\}\big) \cup \bigcup_{i \in \omega} \big( ((\omega+1) \setminus i) \times A_i\big).
    %$$
    $$
        X \coloneq (\omega+1)^2 \setminus \bigcup_{i \in \omega}( i\times A_i). 
    $$
    Let $\sigma \coloneq \tau\upharpoonright X$. It is obvious that $\langle X, \sigma \rangle$ is submetrizable and countable. Let $f \coloneq (f^*\upharpoonright X) \colon X \to S_\omega$. We prove that $f \colon \langle X, \sigma\rangle \to S_\omega$ is continuous compact and sequence-covering but not 1-sequence-covering
    
    It is obvious that $f$ is compact. Let us prove that $f$ is continuous at every point $x \in X$. The only nontrivial case is $x \in \big((\omega+1)\times\{\omega\}\big)$. Let $g$ be an increasing function from $\omega$ to $\omega$. Then $\{0\} \cup {\bf L}(g)$ is a base neighbourhood of $0 = f(x)$. Define $g\prime\colon \omega \to \omega$ by the following rule:
    $$
        g\prime(n) \coloneq f^{-1}_n(g(n))\ \forall n \in \omega.
    $$
    Since $g$ is increasing, from Lemma \ref{tech_lemm} it follows that $g\prime$ is increasing. We prove that
    $$
        f[({\bf L}(g\prime))\cap X] \subseteq \{0\}\cup{\bf L}(g).
    $$
    Take $z \in {\bf L}(g\prime)\cap X$. If $z(2) = \omega$, then $f(z)=0$. Suppose that $x=\langle n,k \rangle$, where $n,k\in\omega$. Then $f(x) = \langle \xi(k), f_{\xi(k)}(k) \rangle$. From the definition of $X$ it follows that $n \geq \xi(k)$. Since $g\prime$ is increasing, we have 
    $$k \geq g\prime(n)\geq g\prime(\xi(k)) = f^{-1}_{\xi(k)}(g(\xi(k))),
    $$ 
    and so $f_{\xi(k)}(k) \geq g(\xi(k))$. Hence $\langle \xi(k), f_{\xi(k)}(k) \rangle \in {\bf U}(\xi(k), g(\xi(k))) \subseteq {\bf L}(g)$.

    Now we prove that $f$ is sequence-covering. Take $p = \big\langle\langle y^{k}_1, y^{k}_2\rangle \big\rangle_{k\in\omega}$ such that $p \in \mathsf{ConvSeq}(\tau_{S_\omega}, 0)$. Let $n \coloneq \mathsf{max}\{y^{k}_1\ \colon\ k\in\omega\}$. For every $k \in \omega$ we have 
    $$
        f^{-1}(\langle y^{k}_1, y^{k}_2\rangle) = \big((\omega+1)\times \{j_k\}\big)\setminus (\xi(j_k)\times A_{\xi(j_k)}) \text{, where } \xi(j_k) = y^{k}_1 \leq n.
    $$
    Hence
    $$
        \forall k\in \omega\ \exists x_k \in f^{-1}(\langle y^{k}_1, y^{k}_2\rangle)\ [x_k(1)=n].
    $$
    And so $\langle x_k \rangle_{k\in\omega} \in \mathsf{ConvSeq}(\sigma, \langle n, \omega \rangle)$ and $f(x_k) = \langle y^{k}_1, y^{k}_2\rangle$ for all $k \in \omega$.

    Now we prove that 
    \begin{gather*}
        \forall y \in (\omega + 1)\times\{\omega\}\ \exists p \in \mathsf{ConvSeq}(\tau_{S_\omega}, 0) \text{ such that}\\ \forall q \in \mathsf{ConvSeq}(\sigma, y)\ [ q \text{ does not cover } p].
    \end{gather*}
    Take $y \in (\omega + 1)\times\{\omega\}$. We need to consider two cases. If $y = \langle \omega, \omega \rangle$, then it is easy to see that $\forall q \in \mathsf{ConvSeq}(\sigma, y)\ [q \text{ is eventually in } (\omega + 1)\times\{\omega\}]$. Now suppose that $y = \langle n, \omega\rangle$, where $n \in \omega$. Consider sequence $\big\langle \langle n+1, k\rangle \big\rangle_{k\in\omega} \in \mathsf{ConvSeq}(\tau_{S_\omega}, 0)$. For every $k \in \omega$ we have 
    $$
        f^{-1}(\langle n+1, k\rangle) = \big((\omega+1)\times \{j_k\}\big)\setminus (\xi(j_k)\times A_{\xi(j_k)}) \text{, where } \xi(j_k) = n+1.
    $$
    Consequently, for every sequence $\langle x_k \rangle_{k\in \omega}$ such that $f(x_k)=\langle n+1, k\rangle$, we have $x_k(1) > n$, and so $\langle x_k \rangle_{k\in \omega} \nin \mathsf{ConvSeq}(\sigma, y)$. And so, it follows that $f$ is not 1-sequence-covering. 
\end{proof}

{\bf Acknowledgement} The author would like to thank Anton Lipin for the help with the simplification of the original proof.

\end{document}